\documentclass{amsproc}

\usepackage{amsmath}

\usepackage{xcolor}

 \usepackage{amssymb}

\usepackage{url}
\usepackage{supertabular} 

\usepackage{mathabx}
\usepackage{multirow} 
\usepackage{graphics}

\usepackage{float}

\usepackage{graphicx}

\allowdisplaybreaks 

\newcommand\scalemath[2]{\scalebox{#1}{\mbox{\ensuremath{\displaystyle #2}}}} 

\newcommand{\ud}{\mathrm{d}}

\renewcommand{\phi}{\varphi}

\newtheorem{theorem}{Theorem}[section]
\newtheorem{lemma}[theorem]{Lemma}
\newtheorem{corollary}[theorem]{Corollary}

\theoremstyle{definition}

\theoremstyle{remark}
\newtheorem{remark}[theorem]{Remark}

\theoremstyle{remark}

\theoremstyle{conjecture}
\newtheorem{conjecture}[theorem]{Conjecture}

\numberwithin{equation}{section}

\title
[Ramanujan--Fine integrals for level $10$]{Ramanujan--Fine integrals for level $10$}
\date{\today}
\author{Shaun Cooper}
\author{Timothy Huber}
\author{Jeffery Opoku}
\address{
Institute of Mathematical and Computational Sciences, Massey University,
Private Bag 102904, North Shore Mail Centre, Auckland, New Zealand
E-mail: s.cooper@massey.ac.nz
}
\address{
School of Mathematical and Statistical Sciences, University of Texas Rio Grande
Valley, Edinburg, TX 78539, USA
E-mail: timothy.huber@utrgv.edu
}
\address{
School of Mathematical and Statistical Sciences, University of Texas Rio Grande
Valley, Edinburg, TX 78539, USA
E-mail: jeffery.opoku01@utrgv.edu
}

\keywords{
Almkvist--Zudilin numbers,
Ap{\'e}ry numbers,
Dedekind eta function,
Domb numbers,
eta quotient,
Ramanujan--Fine integral,
Ramanujan--G{\"o}llnitz--Gordon continued fraction,
Ramanujan's cubic continued fraction,
Ramanujan's lost notebook,
Rogers--Ramanujan continued fraction.
\\
Status: published in Ramanujan J. {\bf 66}, 28 (2025). https://doi.org/10.1007/s11139-024-00995-3}

\usepackage{hyperref}

\hypersetup{
  colorlinks = true,
  urlcolor = blue,
  linkcolor = black,
  citecolor = blue
}

\dedicatory{Dedicated to George Andrews and Bruce Berndt in celebration of their 85th birthdays}

\begin{document}

\begin{abstract}
We investigate the question of when an eta quotient is a derivative of a formal power series with integer coefficients
and present an analysis in the case of level 10. As a consequence, we establish and classify an infinite number
of integral evaluations such as
$$
\int_0^{e^{-2\pi/\sqrt{10}}} q\prod_{j=1}^\infty \frac{(1-q^j)^3(1-q^{10j})^8}{(1-q^{5j})^7} \ud q = \frac14\left(\sqrt{10-4\sqrt{5}}-1\right).
$$
We describe how the results were found and give reasons for why it is reasonable to conjecture that the list is complete for level 10.
\end{abstract}

\maketitle
\section{Introduction}
\label{S:1}
In the Lost Notebook~\cite[p. 46]{lost} S. Ramanujan gave the following identity
\begin{equation}
\label{outrageous}
\cfrac{q^{1/5}}{1+\cfrac{q}{1+\cfrac{q^{2}}{1+\cfrac{q^{3}}{1+\cdots}}}}
= \frac{\sqrt{5}-1}{2} \exp \left\{-\frac15 \int_q^1 \prod_{j=1}^\infty \frac{(1-t^j)^5}{(1-t^{5j})} \,\frac{\ud t}{t}\right\}
\end{equation}
which has been described as ``outrageous''  by G. E. Andrews~\cite{andrews}.
Andrews noted that by a result of L. J. Rogers, the continued fraction on the left hand side has an infinite product expansion given by
$$
q^{1/5}\prod_{j=1}^\infty\frac{(1-q^{5j-4})(1-q^{5j-1})}{(1-q^{5j-3})(1-q^{5j-2})}
$$
and therefore Ramanujan's identity is equivalent by logarithmic differentiation to the nowadays well-known series identity
$$
1-5\sum_{j=1}^{\infty}\left(\frac{j}{5}\right)\frac{jq^{j}}{1-q^{j}} = \prod_{j=1}^{\infty}\frac{(1-q^{j})^{5}}{(1-q^{5j})}, 
\; \text{where} \; \left(\frac{j}{5}\right) = 
\begin{cases} 
1 & \text{if } j \equiv 1 \text{ or } 4\!\! \pmod{5}, \\
-1 & \text{if } j \equiv 2 \text{ or } 3\!\! \pmod{5}, \\
0 & \text{otherwise}
\end{cases}
$$
e.g., see \cite[pp. 257--262]{Part3} or \cite[pp. 310, 350]{cooperbook}.
The value $ \frac{\sqrt{5}-1}{2}$ on the right hand side of Ramanujan's identity comes from exponentiating the constant of integration and is determined
by letting $q\rightarrow 1$ on both sides. A generalisation of~\eqref{outrageous} was selected as one of the ten most fascinating formulas from
Ramanujan's Lost Notebook by G. E. Andrews and B. C. Berndt~\cite{hitparade}.

Another example in this spirit, due to N.~Fine~\cite[pp. 88--90]{fine}, is the surprising integral
\begin{equation}
\label{fine0}
\int_0^{e^{-\pi}} \prod_{j=1}^\infty \frac{(1-q^{2j})^{20}}{(1-q^j)^{16}} \ud q = \frac{1}{16}.
\end{equation}
This is the special case $q=e^{-\pi}$ of the more general result
\begin{equation}
\label{fine1}
\int_0^{q} \prod_{j=1}^\infty \frac{(1-t^{2j})^{20}}{(1-t^j)^{16}} \ud t = q\prod_{j=1}^\infty \frac{(1-q^{4j})^8}{(1-q^j)^8} = \frac{\eta^8(4\tau)}{\eta^8(\tau)}
\end{equation}
where $q=e^{2\pi i \tau}$, and $\eta(\tau)$ is Dedekind's eta function defined by
$$
\eta(\tau) = q^{1/24}\prod_{j=1}^\infty (1-q^j).
$$
The value $1/16$ in~\eqref{fine0} is a simple consequence of substituting the value $\tau=i/2$ into the eta functions on the right hand side of~\eqref{fine1} and
using the functional equation of the Dedekind eta function to compute the value; see \cite[p. 86]{fine}.
By logarithmic differentiation, the identity~\eqref{fine1} is equivalent to 
$$
\prod_{j=1}^\infty \frac{(1-q^{2j})^{20}}{(1-q^j)^8(1-q^{4j})^8} = 1+8\sum_{j= 1}^\infty \frac{jq^j}{1-q^j}-32\sum_{j= 1}^\infty \frac{jq^{4j}}{1-q^{4j}},
$$
a result known as Jacobi's sum of four squares identity, for which many proofs exist, e.g., \cite[pp. 61, 79]{spirit}, \cite[(3.67)]{cooperbook},
\cite{cooperlam}, \cite{hirschhorn4sq}, \cite[pp. 23, 213]{hirschhorn} or~\cite[p. 281]{johnson}.

Fine also gave the identities
\begin{equation}
\label{fine2}
\int_0^{e^{-\pi/\sqrt{2}}} \prod_{n=1}^\infty\frac{(1-q^{2n})^{8}(1-q^{4n})^4}{(1-q^n)^{8}}\, dq
= \frac{1}{\sqrt{32}},
\end{equation}
and
\begin{equation}
\label{fine3}
\int_0^{e^{-\pi/\sqrt{3}}} \prod_{n=1}^\infty\frac{(1-q^{2n})^{14}(1-q^{6n})^6}{(1-q^n)^{8}(1-q^{4n})^8}\, dq
= \frac13.
\end{equation}

In light of these examples, we will refer to an integral
evaluation such as~\eqref{fine0}, \eqref{fine2} and~\eqref{fine3} as a Ramanujan--Fine integral.
Further examples of Ramanujan--Fine integrals have been given in~\cite{aygin}, \cite{coopernonic}, \cite{doyle} and~\cite{zhang}.

By the discussion above, the existence of Ramanujan--Fine integrals relies on there being quotients of eta functions $u(q)$ and $v(q)$ such that
$$
u(q) = q\frac{\ud}{\ud q} v(q)
$$
with the property that $v(e^{-2\pi/\sqrt{N}})$
can be evaluated for some $N \in \mathbb{Z}^+$ or $N \in \mathbb{Q}^+$.
As an example, corresponding to Fine's integral~\eqref{fine0} we have
\begin{align*}
u(q) &= q\prod_{j=1}^\infty \frac{(1-q^{2j})^{20}}{(1-q^j)^{16}} = \frac{\eta^{20}(2\tau)}{\eta^{16}(\tau)}, \\
v(q) &=  q\prod_{j=1}^\infty \frac{(1-q^{4j})^8}{(1-q^j)^8} = \frac{\eta^{8}(4\tau)}{\eta^8(\tau)} \quad\text{and}\quad v(e^{-\pi}) = \frac1{16}.
\end{align*}
Instances of when $u(q)$ and $v(q)$ are quotients of eta functions have been classified by Z.~S.~Aygin and P.~C.~Toh~\cite{aygin}. They found
a total of 203 distinct pairs from level 4 to 36, and for each of those levels their list is complete. The number of such identities of each level is
given in Table~\ref{Table:0}, reproduced exactly from~\cite{aygin}. No further examples were found for other levels less than 100,
so Aygin and Toh conjectured their list is complete.

\begin{table}
\caption{Number of distinct identities for each level. Reproduced from~\cite{aygin}.}
\centering
\begin{tabular}{p{3cm}p{0.5cm}p{0.5cm}p{0.5cm}p{0.5cm}p{0.7cm}p{0.5cm}p{0.5cm}p{0.5cm}p{0.5cm}p{0.5cm}}
\hline
Level & $4$ & $6$ & $8$ & $9$ & $12$ & $16$ & $18$ & $20$ & $24$ & $36$ \\
\hline
No. of identities & 3 & 10 & 4 & 1 & 100 & 4 & 12 & 12 & 32 & 25 \\
\hline
\end{tabular}
\label{Table:0}
\end{table}

In order to find examples for levels not in Table~\ref{Table:0}
we consider the more general question of when an eta quotient is a rational multiple of a derivative of a power series with integer coefficients.
A computer search reveals a huge number of examples, and a complete classification is beyond the scope of this work.

This work is about the specific case of level 10 which does not occur in Aygin and Toh's Table~\ref{Table:0}. 
We find eight examples that fit naturally into two groups of four, along with a separate set of four infinite families.
Analysis of the results involves Ramanujan's level~$10$ function~$k$.
We show that the existence of Ramanujan--Fine integrals reduces to a question about when the indefinite integral
$$
\int \frac{(1-k^2)^{a_1}(1+k-k^2)^{a_2}(1-4k-k^2)^{a_3}}{k^{a_1+a_2+a_3+1}}\, \ud k
$$
is a rational function of $k$. We give conditions on the integers $a_1$, $a_2$ and $a_3$ and prove that the
integral is rational in these cases. Our proof requires properties of the ${}_2F_1$ hypergeometric
series to evaluate a limit which appears to be of independent interest. We suggest in Conjecture~\ref{Conj:1} 
that our list of integers  integers $a_1$, $a_2$ and $a_3$ is complete.

This work is organised as follows.
In Sections~\ref{S:2} and~\ref{S:3} we appeal to the theory of Ap{\'e}ry and related numbers
to survey further integrals that are analogues of Ramanujan's~\eqref{outrageous} and
Fine's~\eqref{fine0}. Some of the results are known and others are new. In Section~\ref{S:4} we describe a computer
search to find Ramanujan--Fine integrals for level~10. Specifically, we describe in detail how
four examples were initially found, prove that the four examples have the claimed properties, and exhibit Ramanujan--Fine integrals in three of the cases.
In Section~\ref{S:5} we use Ramanujan's parameter~$k$ to extend the search to find further examples, including an infinite family.
The proof that the infinite family has the claimed properties 
leads to an interesting limit that is evaluated in Section~\ref{S:6}. Finally in Section~\ref{S:7} we calculate the value of
$k(e^{-2\pi/\sqrt{10}})$ that is required to evaluate Ramanujan--Fine integrals like the example in the abstract.

\bigskip
It will always be assumed that $\tau$ is a complex number with positive imaginary part. Let $q=e^{2\pi i\tau}$
so that $q$ is a complex variable with $|q|<1$. Most of the time we will ignore $\tau$ and work directly with~$q$, but there
will be occasions when we refer back to~$\tau$. Euler's product $E(q)$ is defined by
\begin{equation}
\label{eulerproduct}
E(q) = \prod_{j=1}^\infty (1-q^j).
\end{equation}
Dedekind's eta function is defined by
$$
\eta(\tau) = q^{1/24} \prod_{j=1}^\infty (1-q^j).
$$
It satisfies the well-known functional equation, e.g., see~\cite[p. 70]{hhc} or~\cite[p. 135]{cooperbook},
\begin{equation}
\label{E:eta}
\eta\left(\frac{-1}{\tau}\right) = \sqrt{\frac{\tau}{i}}\,\eta(\tau).
\end{equation}
For any positive integer $N$, let $\eta_N$ be defined by
\begin{equation}
\label{etaN}
\eta_N = \eta(N\tau) =  q^{N/24}E(q^N).
\end{equation}
Thus Fine's example~\eqref{fine0} can be written using either of the two notations as
$$
\int_0^{e^{-\pi}} \frac{E(q^2)^{20}}{E(q)^{16}} \ud q  = \int_0^{e^{-\pi}} \frac{\eta_2^{20}}{\eta_1^{16}} \;\frac{\ud q}{q} = \frac{1}{16}.
$$
When $\eta_N$ notation is used in Ramanujan--Fine integrals, the measure is always~$\frac{\ud q}q$.

\section{Examples from Ap{\'e}ry, Domb and Almkvist--Zudilin numbers}
\label{S:2}
The goal of this section is to quickly exhibit a source of further Ramanujan--Fine integrals by invoking
results from the theory of Ap{\'e}ry, Domb, Almkvist--Zudilin and other related numbers. Some of the integral evaluations are new.
It is also worthwhile to see how some of the known integrals fit in this theory. 

\begin{theorem}
\label{T:1}
Let $r(q)$ be the infinite product representation of the Rogers--Ramanujan continued fraction, so that (e.g., see \cite[pp. 155--160]{spirit} or \cite[pp. 295--302]{cooperbook})
\begin{equation}
\label{rrcf}
r(q) = q^{1/5} \prod_{j=1}^\infty \frac{(1-q^{5j-4})(1-q^{5j-1})}{(1-q^{5j-3})(1-q^{5j-2})}
\end{equation}
and let $\alpha= (1+\sqrt{5})/2$. Let $E(q)$ be defined by~\eqref{eulerproduct}.
Then:
\begin{equation}
\label{Integral1}
\int_0^q \frac{E(t)^{5}}{E(t^5)} r(t)^5 \;\frac{\ud t}{t} = r(q)^5 = \sqrt{\alpha^{10}+1}-\alpha^5 \quad \text{if} \quad q=e^{-2\pi/\sqrt{5}};
\end{equation}

\begin{equation}
\label{Integral2}
\int_0^q \frac{E(t^2)^{8}E(t^3)^6}{E(t)^{10}} \ud t = q\frac{E(q^2)E(q^6)^{5}}{E(q)^{5}E(q^3)} =
\begin{cases}  
\frac{\sqrt{2}}{12} & \text{if} \;\quad q=e^{-2\pi/\sqrt{6}}, \\[8pt]
\frac{\sqrt{3}-1}{24} & \text{if} \;\quad q=e^{-2\pi/\sqrt{3}}, \\[8pt]
\frac{\sqrt{6}-2}{36} & \text{if} \;\quad q=e^{-2\pi/\sqrt{2}}; \\
\end{cases}
\end{equation}

\begin{equation}
\label{Integral3}
\int_0^q \frac{E(t)^{8}E(t^6)^6}{E(t^2)^{10}} \ud t = q\frac{E(q)^{4}E(q^6)^{8}}{E(q^2)^{8}E(q^3)^{4}} =
\begin{cases}  
\frac{(\sqrt{2}-1)^2}{3} & \text{if} \;\quad q=e^{-2\pi/\sqrt{6}}, \\[8pt]
\frac{(2 - \sqrt{3})^2}{3} & \text{if} \;\quad q=e^{-2\pi/\sqrt{3}}, \\[8pt]
\frac{(\sqrt{6}-2)^2}{18} & \text{if} \;\quad q=e^{-2\pi/\sqrt{2}}; \\[8pt]
\end{cases}
\end{equation}

\begin{equation}
\label{Integral4}
\int_0^q \frac{E(t)^{6}E(t^6)^8}{E(t^3)^{10}} \ud t = q\frac{E(q)^3E(q^6)^9}{E(q^2)^3E(q^3)^9} =
\begin{cases}  
\frac{3\sqrt{2}-4}{4} & \text{if} \;\quad q=e^{-2\pi/\sqrt{6}}, \\[8pt]
\frac{(\sqrt{3}-1)^3}{16} & \text{if} \;\quad q=e^{-2\pi/\sqrt{3}}, \\[8pt]
\frac{(\sqrt{6}-2)^3}{8} & \text{if} \;\quad q=e^{-2\pi/\sqrt{2}}; \\[8pt]
\end{cases}
\end{equation}

\begin{equation}
\label{Integral5}
\int_0^q \frac{E(t^2)^{8}E(t^4)^4}{E(t)^{8}} \ud t = q\frac{E(q^2)^2E(q^8)^4}{E(q)^4E(q^4)^2} = \frac{1}{\sqrt{32}}
\quad\text{when}\quad q=e^{-\pi/\sqrt{2}}; \\[5pt]
\end{equation}

\begin{equation}
\label{Integral6}
\int_0^q \frac{E(t^3)^{10}}{E(t)^6} \ud t = q\frac{E(q^9)^3}{E(q)^3} = \frac{1}{3\sqrt{3}}\quad\text{when}\quad q=e^{-2\pi/3}. \\[5pt]
\end{equation}

\end{theorem}

We offer some remarks before starting the proof.
\begin{enumerate}
\item
The integral~\eqref{Integral1} is equivalent to Ramanujan's example in~\eqref{outrageous} in the sense
that that applying logarithmic differentiation to each integral results in the same series identity. The appeal
of Ramanujan's form~\eqref{outrageous} is best described by Andrews~\cite{andrews} who wrote ``Ramanujan has a flair for finding the most astounding version
of any particular identity''.
\item
The integral~\eqref{Integral4} has an equivalent formula in the style of Ramanujan's~\eqref{outrageous} given by
$$
 \cfrac{q^{1/3}}{1+\cfrac{q+q^{2}}{1+\cfrac{q^{2}+q^{4}}{1+\cfrac{q^{3}+q^{6}}{1+\cdots}}}}
 = \frac12 \exp \left\{-\frac13 \int_q^1 \prod_{j=1}^\infty \frac{(1-t^j)^3(1-t^{2j})^3}{(1-t^{3j})(1-t^{6j})} \,\frac{\ud t}{t}\right\}
$$
where the left hand side is Ramanujan's cubic continued fraction. 
This particular integral does not seem to have been mentioned before in previous works
such as~\cite{madhu} and~\cite{ahlgren}. All of the necessary formulas needed to
prove this identity are well-known and can be found, for example, in~\cite[Ch. 6]{cooperbook}.
The corresponding integral for the Ramanujan--G{\"o}llnitz--Gordon continued fraction is
\begin{align*}
& \quad\cfrac{q^{1/2}}{1+q+\cfrac{q^{2}}{1+q^3+\cfrac{q^{4}}{1+q^5+\cdots}}} \qquad\qquad \\
&\qquad\qquad =(\sqrt{2}-1) \exp \left\{-\frac12 \int_q^1 \prod_{j=1}^\infty \frac{(1-t^j)^2(1-t^{2j})(1-t^{4j})^3}{(1-t^{8j})^2} \,\frac{\ud t}{t}\right\},
 \end{align*}
due to S.~Ahlgren et al.~\cite[(4.16)]{ahlgren}.
They also gave the level~12 example
\begin{align*}
\qquad& q\prod_{j=1}^\infty \frac{(1-q^{12j-11})(1-q^{12j-1})}{(1-q^{12j-7})(1-q^{12j-5})}\qquad \\
& =(2-\sqrt{3}) \exp \left\{- \int_q^1 \prod_{j=1}^\infty \frac{(1-t^j)(1-t^{3j})(1-t^{4j})^2(1-t^{6j})^2}{(1-t^{12j})^2} \,\frac{\ud t}{t}\right\}.
 \end{align*}
A continued fraction for the left hand side has been given by M. S. Mahadeva Naika et al.,~\cite{naika}.
Further properties have been studied in~\cite{ksw2009},~\cite[Ch. 12]{cooperbook}, 
\cite{cooperye12},
\cite{lee2},
\cite{lin},
\cite{naika2012} 
and~\cite{vasuki}.

\item
Integrals \eqref{Integral2} and \eqref{Integral4} in the case $q=e^{-2\pi/\sqrt{6}}$ correspond to the case $j=1$ of the
results~\cite[(6.8) and (6.9)]{aygin} of Aygin and Toh. The integral~\eqref{Integral3} and the other values of $q$ in~\eqref{Integral2}--\eqref{Integral4}
were not considered in~\cite{aygin}. As we will see in the proof, other limits of integration are possible for all of~\eqref{Integral1}--\eqref{Integral6}.
\item
The integrals in~\cite{aygin} and~\cite{zhang} that involve the positive integer~$j$ are all obtained from the $j=1$ case by
the chain rule for differentiating a power. Therefore we do not consider those generalisations in this work; we will
focus solely on the $j=1$ case to emphasise the essence of each result.
\item
The integral~\eqref{Integral5} is Fine's example~\eqref{fine2}, and the integral~\eqref{Integral6} is from~\cite{coopernonic}.
\end{enumerate}

\begin{proof}[Proof of Theorem~\ref{T:1}]
The proofs of~\eqref{Integral1}--\eqref{Integral6} are procedurally the same and consist of extracting information
from Table~\ref{Table:1}. Each row of the table consists of three functions of $q$ denoted by $x$, $w$ and $y$, where
$w$ and $y$ are given in terms of $x$ by the identities
$$
w=\frac{x}{1-\alpha x - \gamma x^2} \quad\text{and}\quad y=q\frac{\ud}{\ud q} \log x
$$ 
where $\alpha$ and $\gamma$ are as in the table. There is also a parameter $\beta$ in Table~\ref{Table:1} which we will explain next.
The function $y$ may be expanded in powers of $w$ by the formula
$$
y=\sum_{n=0}^\infty s(n) w^n
$$
where the coefficients $s(n)$ are given by the respective binomial sums in the table, and satisfy the three-term recurrence relation
$$
(n+1)^3s(n+1)=-(2n+1)(\alpha n^2+\alpha n+\alpha-2\beta)s(n)
 -(\alpha^2+4\gamma) n^3 s(n-1)
$$
and initial condition $s(0)=1$. The numbers $s(n)$ in the table corresponding to 6(A), 6(B) and 6(C) are called Ap{\'e}ry numbers, Domb numbers, and
Almkvist--Zudilin numbers, respectively. Not all of this information is needed in the proof, but it is included for completeness because it is useful to
know it is part of a bigger theory. For further details, see~\cite[pp. 396--402]{cooperbook}.

With this information, let us consider the integral~\eqref{Integral1}. From the row of the table corresponding to level~5, we have
$$
y=q\frac{\ud}{\ud q} \log x \quad\text{where}\quad x=r^5(q) \quad\text{and}\quad y=\dfrac{\eta_1^5}{\eta_5}
$$
where $r(q)$ is as in~\eqref{rrcf} and $\eta_1$ and $\eta_5$ are given by~\eqref{etaN}. Integration of the differential equation, and noting that $x(0) = r^5(0) = 0$, gives
$$
\int_0^q x(t) y(t) \, \frac{\ud t}{t} = x(q)
$$
and this proves the first part of~\eqref{Integral1}. The value of $r(q)$ when $q=e^{2\pi/\sqrt{5}}$ was known to Ramanujan, e.g., see~\cite[p. 307]{cooperbook},
and substitution of this value of $q$ completes the proof of~\eqref{Integral1}.

Many other evaluations of $r(q)$ are known and so different values of $q$ can be used in the definite integral. The value $q=e^{2\pi/\sqrt{5}}$ is the simplest interesting case.

The other integrals~\eqref{Integral2}--\eqref{Integral6} are extracted from the other rows of Table~\ref{Table:1} by the same process. The specific values of~\eqref{Integral2}--\eqref{Integral4} 
are from\footnote{The value of $r_b(e^{-2\pi/\sqrt{2}})$ in \cite[p. 412]{cooperbook} contains a misprint and should be $(\sqrt{3}-\sqrt{2})/12$. Remember
to use $x_a$, $x_b$ and $x_c$ as given by~\cite[(6.5)]{cooperbook}, i.e., $x_a=r_a$,  $x_b=r_b^2$ and $x_c=r_c^3$.}~\cite[p. 412]{cooperbook}.
We have already mentioned that the integral~\eqref{Integral5} is due to Fine~\cite{fine}, and~\eqref{Integral6} is from~\cite{coopernonic}.

\end{proof}

\begin{table}
\caption{Data for Theorem~\ref{T:1}}
{\renewcommand{\arraystretch}{2.9}r
\centering
{\resizebox{13cm}{!}{
\begin{tabular}{|c|c|c|c|c|c|}
\hline
level & $x$ & $w$ & $y$ & $(\alpha,\beta,\gamma)$ & $s(n)$ \\
\hline \hline
$5$ 
& $r(q)^5$ 
& $\displaystyle{\frac{\eta_5^{6}}{\eta_1^{6}}}$
& $\dfrac{\eta_1^5}{\eta_5}$ 
&$ (11,3,1)$
& $\displaystyle{\sum_{j} (-1)^{j+n}{n \choose j}^{3} {4n-5j \choose 3n}}$
\\
\hline
$6$ (A)
& $\dfrac{\eta_2\eta_6^5}{\eta_1^5\eta_3}$ 
& $\displaystyle{\frac{\eta_1^{12}\eta_6^{12}}{\eta_2^{12}\eta_3^{12}}}$
& $\dfrac{\eta_2^7\eta_3^7}{\eta_1^5\eta_6^5}$ 
& $(-17,-6,-72)$
&  $\displaystyle{\sum_{j} {n \choose j}^{2}{n+j \choose j}^{2}}$
\\
\hline
$6$ (B)
& $\dfrac{\eta_1^4\eta_6^8}{\eta_2^8\eta_3^4}$ 
& $\displaystyle{\frac{\eta_2^{6}\eta_6^{6}}{\eta_1^{6}\eta_3^{6}}}$
& $\dfrac{\eta_1^4\eta_3^4}{\eta_2^2\eta_6^2}$ 
& $(10,3,-9)$
& $\displaystyle{(-1)^{n}\sum_{j} {n \choose j}^{2}{2j \choose j}{2n-2j \choose n-j}}$
\\
\hline
$6$ (C)
& $\dfrac{\eta_1^3\eta_6^9}{\eta_2^3\eta_3^9}$ 
& $\displaystyle{\frac{\eta_3^{4}\eta_6^{4}}{\eta_1^{4}\eta_2^{4}}}$
& $\dfrac{\eta_1^3\eta_2^3}{\eta_3\eta_6}$ 
& $(7,2,8)$ 
& $\displaystyle{\sum_{j} (-3)^{n-3j}{n+j \choose j,j,j,j,n-3j}}$ 
\\
\hline
$8$ 
& $\dfrac{\eta_2^2\eta_8^4}{\eta_1^4\eta_4^2}$ 
& $\displaystyle{\frac{\eta_1^{8}\eta_8^{8}}{\eta_2^{8}\eta_4^{8}}}$
& $\dfrac{\eta_2^6\eta_4^6}{\eta_1^4\eta_8^4}$
& $(12,4,-32)$
& $\displaystyle{\sum_{j} {n \choose j}^{2}{2j \choose n}^{2}}$
\\
\hline
$9$ 
& $\dfrac{\eta_9^3}{\eta_1^3}$  
& $\displaystyle{\frac{\eta_1^{6}\eta_9^{6}}{\eta_3^{12}}}$
& $\dfrac{\eta_3^{10}}{\eta_1^3\eta_9^3}$ 
& $(-9,-3,-27)$
& $\displaystyle{\sum_{j,\ell} {n \choose j}^{2}{n \choose \ell}{j \choose \ell}{j+\ell \choose n}}$
\\
\hline
\end{tabular}}}}
\label{Table:1}
\end{table}

\section{Further examples}
\label{S:3}
Additional examples can be created using data from~\cite[pp. 402--403]{cooperbook} and~\cite[Table 3]{4term}, summarised in
Table~\ref{Table:2}. In that table, the key relations are
\begin{equation}
\label{dwdq}
q\frac{\ud}{\ud q} \log w = Z 
\end{equation}
and
$$
Z=\sum_{n=0}^\infty T(n)X^n.
$$
For levels 5, 6, 8 and 9, the function $w$ in Table~\ref{Table:2} is identical to the
corresponding function $w$ in Table~\ref{Table:1}. 
Moreover for those levels, we have the relation, e.g., see \cite[(6.72)]{cooperbook},
$$
X = \frac{w}{1+2\alpha w + (\alpha^2+4\gamma)w^2}
$$
where $\alpha$ and $\gamma$ are as for Table~\ref{Table:1}.
Details of proofs of the identities for levels $1,\ldots,6$, 8 and 9 in Table~\ref{Table:2} are provided in~\cite[pp. 400--403]{cooperbook}.
Proofs for level 7 are in~\cite[pp. 449, 458]{cooperbook}. The function $w$ in Table~\ref{Table:2} for level~10 is different from the
corresponding function $w$ in~\cite{4term}. The function $w$ in Table~\ref{Table:2} has been chosen so that property~\eqref{dwdq} holds; see~\cite[p. 538]{cooperbook}.

In principle, integral evaluations can be written for each row of data in Table~\ref{Table:2}, by integrating~\eqref{dwdq}.
As  example, for level 7, the result is
$$
\int_0^q \frac{E(t^7)^6}{E(t)^2} \left(\frac{E(t)^4}{E(t^7)^4}+13t+49t^2\frac{E(t^7)^4}{E(t)^4}\right)^{2/3} \, \ud t = q\frac{E(q^7)^4}{E(q)^4} 
$$
and both sides are equal to $1/7$ when $q=e^{-2\pi/\sqrt{7}}$. The analogue of this result corresponding to Ramanujan's~\eqref{outrageous}
has been given in~\cite[(4.7)]{ahlgren}.

We also mention the example corresponding to level 4. By~\cite[Theorem 4.1]{4term}, another formula for $Z$ in Table~\ref{Table:2} is 
$$
Z = \frac{\eta_2^{20}}{\eta_1^8\eta_4^8}.
$$
Hence, integration of \eqref{dwdq} with data from Table~\ref{Table:2} for level 4 leads to
$$
\int_0^{q} \frac{E(t^2)^{20}}{E(t)^{16}} \ud t = q\,\frac{E(q^4)^8}{E(q)^8}
$$
and this is Fine's example in~\eqref{fine0} and~\eqref{fine1}.

In the next section we will proceed in a completely different direction and show how to find Ramanujan--Fine integrals involving
eta quotients for level 10.

\begin{table}
\caption{Data for $\displaystyle{Z=\sum_{n=0}^\infty T(n)X^n}. \qquad$
The notation is: \\
$\displaystyle{\scalemath{.9}{Q=1+240\sum_{j=1}^\infty \frac{j^3 q^{j}}{1-q^j},}}\quad$ 
$\displaystyle{\scalemath{.9}{R=1-504\sum_{j=1}^\infty \frac{j^5q^{j}}{1-q^j},}}$}
{\renewcommand{\arraystretch}{2.9}
\centering
{\resizebox{13cm}{!}{
\begin{tabular}{|c|c|c|c|c|}
\hline
level & $w$ & $X$ & $Z$ & $T(n)$ \\
\hline \hline
$1$ 
& $\frac{1}{432}\left(\frac{Q^{3/2}-R}{Q^{3/2}+R}\right)$ 
& $\displaystyle{\frac{w}{(1+432w)^2}}$ 
& $\displaystyle{\frac{\eta_1^{4}}{X^{1/6}}}$
& $\displaystyle{{6n \choose 3n}{3n \choose n}{2n \choose n}}$ 
 \\
\hline
$2$ 
& $\displaystyle{\frac{\eta_2^{24}}{\eta_1^{24}}}$
& $\displaystyle{\frac{w}{(1+64w)^2}}$ 
& $\displaystyle{\frac{\eta_1^{2}\eta_2^2}{X^{1/4}}}$ 
& $\displaystyle{{4n \choose 2n}{2n \choose n}^2}$ 
\\
\hline
$3$ 
& $\displaystyle{\frac{\eta_3^{12}}{\eta_1^{12}}}$
& $\displaystyle{\frac{w}{(1+27w)^2}}$ 
& $\displaystyle{\frac{\eta_1^{2}\eta_3^2}{X^{1/3}}}$ 
& $\displaystyle{{3n \choose n}{2n \choose n}^2}$ 
\\
\hline
$4$ 
& $\displaystyle{\frac{\eta_4^{8}}{\eta_1^{8}}}$
& $\displaystyle{\frac{w}{(1+16w)^2}}$ 
& $\displaystyle{\frac{\eta_1^2\eta_4^2}{X^{5/12}}}$ 
& $\displaystyle{{2n \choose n}^3}$ 
\\
\hline
$5$ 
& $\displaystyle{\frac{\eta_5^{6}}{\eta_1^{6}}}$
& $\displaystyle{\frac{w}{1+22w+125w^2}}$ 
& $\displaystyle{\frac{\eta_1^{2}\eta_5^2}{X^{1/2}}}$ 
& $\displaystyle{{2n \choose n}\sum_{j} {n \choose j}^{2}{n+j \choose j}}$ 
\\
\hline
$6$ (A)
& $\displaystyle{\frac{\eta_1^{12}\eta_6^{12}}{\eta_2^{12}\eta_3^{12}}}$
& $\displaystyle{\frac{w}{1-34w+w^2}}$ 
& $\displaystyle{\frac{\eta_1\eta_2\eta_3\eta_6}{X^{1/2}}}$ 
&  $\displaystyle{{2n \choose n}\sum_{j,\ell}(-8)^{n-j}{n \choose j}{j \choose \ell}^{3}}$ 
\\
\hline
$6$ (B)
& $\displaystyle{\frac{\eta_2^{6}\eta_6^{6}}{\eta_1^{6}\eta_3^{6}}}$
& $\displaystyle{\frac{w}{1+20w+64w^2}}$ 
& $\displaystyle{\frac{\eta_1\eta_2\eta_3\eta_6}{X^{1/2}}}$ 
& $\displaystyle{{2n \choose n}\sum_{j} {n \choose j}^{2}{2j \choose j}}$
\\
\hline
$6$ (C)
& $\displaystyle{\frac{\eta_3^{4}\eta_6^{4}}{\eta_1^{4}\eta_2^{4}}}$
& $\displaystyle{\frac{w}{1+14w+81w^2}}$ 
& $\displaystyle{\frac{\eta_1\eta_2\eta_3\eta_6}{X^{1/2}}}$ 
& $\displaystyle{{2n \choose n}\sum_{j} {n \choose j}^{3}}$ 
\\
\hline
$7$ 
& $\displaystyle{\frac{\eta_7^{4}}{\eta_1^{4}}}$
& $\displaystyle{\frac{w}{1+13w+49w^2}}$ 
& $\displaystyle{\frac{\eta_1^{2}\eta_7^2}{X^{2/3}}}$ 
& $\displaystyle{\sum_{j} {n \choose j}^{2} {2j \choose n} {n+j \choose j} }$ 
\\
\hline
$8$ 
& $\displaystyle{\frac{\eta_1^{8}\eta_8^{8}}{\eta_2^{8}\eta_4^{8}}}$
& $\displaystyle{\frac{1}{1-24w+16w^2}}$ 
& $\displaystyle{\frac{\eta_2^{2}\eta_4^2}{X^{1/2}}}$ 
& $\;\;\displaystyle{{2n \choose n}(-1)^{n}\sum_j {n \choose j}{2j \choose j}{2n-2j \choose n-j}}\;\;$ 
\\
\hline
$9$ 
& $\displaystyle{\frac{\eta_1^{6}\eta_9^{6}}{\eta_3^{12}}}$
& $\displaystyle{\frac{w}{1-18w-27w^2}}$ 
& $\displaystyle{\frac{\eta_3^4}{X^{1/2}}}$ 
& $\displaystyle{{2n \choose n}{\sum_{j}} (-3)^{n-3j} {n \choose j}{n-j \choose j}{n-2j \choose j}}$ 
\\
\hline
$10$ 
& $\displaystyle{\frac{\eta_5^{2}\eta_{10}^{2}}{\eta_1^{2}\eta_2^{2}}}$
& $\displaystyle{\frac{w}{1+6w+25w^2}}$ 
& $\displaystyle{\frac{\eta_1\eta_2\eta_5\eta_{10}}{X^{3/4}}}$ 
& $\displaystyle{\sum_j {n \choose j}^4}$  \\
\hline
\end{tabular}}}}
\vspace{2in}
\label{Table:2}
\end{table}

\section{A computer search for eta quotients: initial results}
\label{S:4}
In this section we will describe how Ramanujan--Fine integrals can be found for level~10.
Taking as prototypes Fine's integrals and the integrals~\eqref{Integral2}--\eqref{Integral6} suggests
looking for pairs of functions $u(q)$ and $v(q)$ with
$$
\int u(q) \frac{\ud q}{q} = v(q), \quad\text{that is,}\quad u(q)= q\frac{\ud}{\ud q} v(q).
$$
In the above mentioned examples, both $u(q)$ and $v(q)$ are eta quotients.
We investigate the more general situation when $u(q)$ is an eta quotient and $v(q)$
is a power series with integer coefficients.
The eta quotients $u(q)$ in Fine's examples are all of the form
$$
\prod_{d|N} \eta_d^{e_d}
$$
where $N$ is a fixed positive integer, $d$ ranges over the divisors of $N$, and the exponents $e_d$ are integers satisfying
$$
\sum_d e_d = 4 \quad\text{and}\quad \sum_d d e_d = 24.
$$
With these examples in mind, we conducted a computer search for exponents $e_1$, $e_2$, $e_5$ and $e_{10}$ that satisfy:
\begin{enumerate}
\item $e_1+e_2+e_5+e_{10}= 4$;
\item $e_1+2e_2+5e_5+10e_{10} \equiv 0\pmod{24}$; and
\item the eta quotient $u(q) = \eta_1^{e_1}\eta_2^{e_2}\eta_5^{e_5}\eta_{10}^{e_{10}}$ is the derivative of a series 
$$
v(q) = \sum_{j=j_0}^\infty c(j)q^j \quad \text{for some $j_0 \in \mathbb{Z}$}
$$
where the coefficients $c(j)$ are all integers.
\end{enumerate}
We quickly found the examples
$$
(e_1,e_2,e_5,e_{10}) = (8,-7,0,3)\quad\text{and}\quad (0,3,8,-7)
$$
corresponding to
\begin{align}
\frac{\eta_1^8\eta_{10}^3}{\eta_2^7} &=
q\frac{\ud}{\ud q}\left(q-4q^2 + 9q^3 - 14q^4 + 21q^5-36q^6+58q^7-84q^8+\cdots\right) \label{e10a}
\intertext{and}
\frac{\eta_2^3\eta_{5}^8}{\eta_{10}^7} &=
-q\frac{\ud}{\ud q}\left(\frac{1}{q}+3q+2q^4 -q^5-4q^6-3q^9+4q^{10}+8q^{11}+\cdots\right). \label{e10b}
\end{align}
The involutions obtained by replacing $(e_1,e_2,e_5,e_{10})$ by $(e_{10},e_5,e_2,e_{1})$ in these examples were not found by the search, because
\begin{align}
\frac{\eta_1^3\eta_{10}^8}{\eta_5^7} &=
q\frac{\ud}{\ud q}\left(\frac12q^2-q^3+q^5+q^7-\frac{7}{2}q^8+\frac72q^{10}+3q^{12}-10q^{13}+\cdots\right) \label{e10c}
\intertext{and}
\frac{\eta_2^8\eta_{5}^3}{\eta_{1}^7} &=
q\frac{\ud}{\ud q}\left(q+\frac72q^2+9q^3+21q^4+46q^5+94q^6+183q^7+\frac{687}{2}q^8+\cdots\right) \label{e10d}
\end{align}
and the coefficients are not all integers.
We will show in Corollary~\ref{integercoeffs} that the coefficients in parentheses on the right hand sides of~\eqref{e10a} and~\eqref{e10b} are all integers,
and two times the corresponding coefficients in~\eqref{e10c} and~\eqref{e10d} are also integers.

In order to analyse the eta quotients in~\eqref{e10a}--\eqref{e10d} we recall Ramanujan's
parameter~$k$ defined by
\begin{equation}
\label{kdef}
k(q) = r(q) r^2(q^2) = q\prod_{j=1}^\infty \frac{(1-q^{10j-9})(1-q^{10j-8})(1-q^{10j-2})(1-q^{10j-1})}{(1-q^{10j-7})(1-q^{10j-6})(1-q^{10j-4})(1-q^{10j-3})}.
\end{equation}
By~\cite[Theorems 10.7 and 10.8]{cooperbook} we have
\begin{align}
\eta_1^{24}&= y_{10}^6\;\frac{k(1-4k-k^2)^4}{(1-k^2)^4(1+k-k^2)}, \label{rp1} \\
\eta_2^{24}&= y_{10}^6\;\frac{k^2(1+k-k^2)^4}{(1-k^2)^5(1-4k-k^2)}, \label{rp2} \\
\eta_5^{24}&= y_{10}^6\;\frac{k^5(1-k^2)^4}{(1+k-k^2)^5(1-4k-k^2)^4}, \label{rp3} \\
\eta_{10}^{24}&= y_{10}^6\;\frac{k^{10}}{(1-k^2)(1+k-k^2)^4(1-4k-k^2)^5} \label{rp4}
\end{align}
where
$$
y_{10} = q\frac{\ud}{\ud q} \log k =\frac{\eta_1\eta_2^2\eta_5^3}{\eta_{10}^2}.
$$
Therefore, the eta quotients~\eqref{e10a}--\eqref{e10d} can be parameterised in terms of $k$ as
\begin{align*}
\frac{\eta_1^8\eta_{10}^3}{\eta_2^7} 
&= \frac{k(1-4k-k^2)}{(1+k-k^2)^2} \,q\frac{\ud}{\ud q} \log k, \\
\frac{\eta_2^3\eta_{5}^8}{\eta_{10}^7} 
&= \frac{1-k^2}{k} \,q\frac{\ud}{\ud q} \log k, \\
\frac{\eta_1^3\eta_{10}^8}{\eta_5^7} 
&= \frac{k^2}{(1-k^2)^2} \,q\frac{\ud}{\ud q} \log k, \\
\frac{\eta_2^8\eta_{5}^3}{\eta_1^7} 
&= \frac{k(1+k-k^2)}{(1-4k-k^2)^2} \,q\frac{\ud}{\ud q} \log k.
\end{align*}
Hence, by the chain rule and simple calculus we obtain
\begin{align}
\frac{\eta_1^8\eta_{10}^3}{\eta_2^7} 
&= -q\frac{\ud}{\ud q} \left(\frac{1+k^2}{1+k-k^2}\right), \label{10a1}\\
\frac{\eta_2^3\eta_{5}^8}{\eta_{10}^7} 
&= -q\frac{\ud}{\ud q} \left(\frac{1+k^2}{k}\right), \label{10a2}\\
\frac{\eta_1^3\eta_{10}^8}{\eta_5^7} 
&= \frac14\,q\frac{\ud}{\ud q} \left(\frac{1+k^2}{1-k^2}\right) \label{10a3} \intertext{and}
\frac{\eta_2^8\eta_{5}^3}{\eta_1^7} 
&= \frac14\,q\frac{\ud}{\ud q} \left(\frac{1+k^2}{1-4k-k^2}\right).\label{10a4}
\end{align}
Any of~\eqref{10a1}, \eqref{10a3} or~\eqref{10a4} may be integrated over the interval $0<q<q_0$ to produce a definite integral that can be evaluated exactly
provided the value of $k(q_0)$ can be determined. Note that~\eqref{10a2} is not integrable at $q=0$. In this way we obtain the following definite integrals.
\begin{theorem}
\label{10point2}
The following integral evaluations hold:
$$
\int_0^{e^{-2\pi/\sqrt{10}}} \frac{E(q)^8E(q^{10})^3}{E(q^2)^7} \ud q = 1-2\,\sqrt{\frac{10-4\sqrt{5}}{5}},
$$
$$
\int_0^{e^{-2\pi/\sqrt{10}}} \frac{qE(q)^3E(q^{10})^8}{E(q^5)^7} \ud q = \frac14\left(\sqrt{10-4\sqrt{5}}-1\right),
$$
and
$$
\int_0^{e^{-2\pi/\sqrt{10}}} \frac{E(q^2)^8E(q^5)^3}{E(q)^7} \ud q = \frac14\left(\sqrt{\frac{10+4\sqrt{5}}{5}}-1\right).
$$
\end{theorem}
\begin{proof}
These follow by integrating~\eqref{10a1}, \eqref{10a3} and~\eqref{10a4}, respectively, and using the values
$$
\left[ \frac{1}{k}-k \right]_{q=e^{-2\pi/\sqrt{10}}} \; = 4+2\sqrt{5}
$$
and hence 
$$
k(e^{-2\pi/\sqrt{10}}) = \sqrt{10+4\sqrt{5}} - 2 - \sqrt{5}.
$$
We defer a proof of these evaluations to the appendix; see ~\eqref{u} and~\eqref{keval}.
\end{proof}

\noindent
We are now ready to analyse the integrality properties of the coefficients \mbox{in~\eqref{e10a}--\eqref{e10d}.}
\begin{corollary}
\label{integercoeffs}
The coefficients $a(n)$, $b(n)$, $c(n)$ and $d(n)$ defined by
\begin{align*}
\frac{\eta_1^8\eta_{10}^3}{\eta_2^7} &= q\frac{\ud}{\ud q} \left(\sum_{n=1}^\infty a(n)q^n\right) = q\frac{\ud}{\ud q} \left( q-4q^2+9q^3-14q^4+\cdots\right),\\
\frac{\eta_2^3\eta_{5}^8}{\eta_{10}^7} &= q\frac{\ud}{\ud q} \left(\sum_{n=-1}^\infty b(n)q^n\right) = q\frac{\ud}{\ud q} \left( \frac{-1}{q}-3q-2q^4+q^5+\cdots\right),\\
2\frac{\eta_1^3\eta_{10}^8}{\eta_5^7} &= q\frac{\ud}{\ud q} \left(\sum_{n=2}^\infty c(n)q^n\right) = q\frac{\ud}{\ud q} \left( q^2-2q^3+2q^5+2q^7-7q^8+\cdots\right),\\
2\frac{\eta_2^8\eta_{5}^3}{\eta_{1}^7} &= q\frac{\ud}{\ud q} \left(\sum_{n=1}^\infty d(n)q^n\right) = q\frac{\ud}{\ud q} \left( 2q+7q^2+18q^3+42q^4+\cdots\right),\\
\end{align*}
are all integers.
\end{corollary}
\begin{proof}
We only prove the last property, because the other proofs are similar. By~\eqref{10a4} we have
\begin{align*}
2\frac{\eta_2^8\eta_{5}^3}{\eta_{1}^7}  
&=\frac12\,q\frac{\ud}{\ud q} \left(\frac{1+k^2}{1-4k-k^2}\right) \\
&=\frac12\,q\frac{\ud}{\ud q} \left(1+\frac{4k+2k^2}{1-4k-k^2}\right) \\
&=q\frac{\ud}{\ud q} \left(\frac{2k+k^2}{1-4k-k^2}\right).
\end{align*}
By the definition of $k$ in~\eqref{kdef} it follows that
$$
k(q) = \sum_{n=1}^\infty t(n)q^n
$$
where the coefficients $t(n)$ are all integers. Hence, the coefficients in the
expansion of $\frac{2k+k^2}{1-4k-k^2}$ are all integers. In other words, the coefficients
$d(n)$ are all integers.
\end{proof}

\section{A computer search for eta quotients: further results}
\label{S:5}
Based on the findings in the previous section, we can extend the search by considering
\begin{enumerate}
\item $e_1+e_2+e_5+e_{10}= 4$;
\item $e_1+2e_2+5e_5+10e_{10} \equiv 0\pmod{24}$; and
\item an integer multiple of the eta quotient $u(q) = \eta_1^{e_1}\eta_2^{e_2}\eta_5^{e_5}\eta_{10}^{e_{10}}$ is the derivative of a series, i.e., 
$$
b\, \eta_1^{e_1}\eta_2^{e_2}\eta_5^{e_5}\eta_{10}^{e_{10}} = \sum_{j=j_0}^\infty jc(j)q^j \quad \text{for some $j_0 \in \mathbb{Z}$}
$$
for some integer $b$ and where the coefficients $c(j)$ are all integers.
\end{enumerate}
We typically take $b$ to be a product of small primes, e.g., $b=2\times 3\times 5 \times 7=210$.
By applying the computer search and reasoning as in the previous section, we end up with the data in Tables~\ref{Table:8} and~\ref{Table:9}.
The notation for those tables is:
$$
\eta_1^{e_1}\eta_2^{e_2}\eta_5^{e_5}\eta_{10}^{e_{10}} = f(k) \times q\frac{\ud}{\ud q} \log k = \frac{\eta_1\eta_2^2\eta_5^3}{\eta_{10}^2}\,f(k)
$$
where
$$
f(k)= \eta_1^{e_1-1}\eta_2^{e_2-2}\eta_5^{e_5-3}\eta_{10}^{e_{10}+2} = k^{a_0}(1-k^2)^{a_1}(1+k-k^2)^{a_2}(1-4k-k^2)^{a_3} 
$$
and where the exponents $a_0$, $a_1$, $a_2$ and $a_3$ are given by
\begin{align*}
a_0&= \frac{1}{24}\left(e_1+2e_2+5e_5+10e_{10}\right), \\
a_1 &= \frac{1}{24}\left(-4e_1-5e_2+4e_5-e_{10}\right), \\
a_2 &= \frac{1}{24}\left(-e_1+4e_2-5e_5-4e_{10}\right), \\
a_3 &= \frac{1}{24}\left(4e_1-e_2-4e_5-5e_{10}\right).
\end{align*}
As a check we have
$$
a_0+a_1+a_2+a_3=0 \quad\text{and}\quad e_1+e_2+e_5+e_{10}=4.
$$

\begin{table}[ht]
\caption{Data for Conjecture~\ref{Conj:1} and Theorem~\ref{T:5}. }
{\renewcommand{\arraystretch}{2.4}
\centering
\begin{tabular}{|c||c|c|c|c||c|c|c|c||c|}
\hline
label & $e_1$ & $e_2$ & $e_5$ & $e_{10}$ & $a_0$ & $a_1$ & $a_2$ & $a_3$ &
 $\displaystyle{\int \eta_1^{e_1}\eta_2^{e_2}\eta_5^{e_5}\eta_{10}^{e_{10}} \, \frac{\ud q}{q} = \int f(k) \frac{\ud k}{k}} = g(k)$ \\
\hline \hline
1.0 &8 & -7 & 0 & 3 & 1 & 0 & -2 & 1 &$\displaystyle{-\left(\frac{1+k^2}{1+k-k^2}\right)}$   \\  \hline
2.0 &0 & 3 & 8 & -7 & -1 & 1 & 0 & 0 & $\displaystyle{-\left(\frac{1+k^2}{k}\right)}$ \\ \hline 
3.0 & 3 & 0 & -7 & 8 & 2 & -2 & 0 & 0 & $\displaystyle{\frac14\left(\frac{1+k^2}{1-k^2}\right)}$  \\ \hline
4.0 &-7 & 8 & 3 & 0 & 1 & 0 & 1 & -2 & $\displaystyle{\frac14\left(\frac{1+k^2}{1-4k-k^2}\right)}$  \\ \hline \hline
5 &8 & -13 & 0 & 9 & 3 & 1 & -4 & 0 & $\displaystyle{\frac{2+6k-5k^3}{15(1+k-k^2)^3}}$  \\ \hline
6 &0 & 9 & 8 & -13 & -3 & 0 & 2 & 1& $\displaystyle{\frac{-(1+k^2)(1-3k-31k^2+3k^3+k^4)}{3k^3}}$ \\ \hline
7 &9 & 0 & -13 & 8 & 1 & -4 & 1 & 2 & $\displaystyle{\frac{-29+6k+66k^2-21k^4-6k^5}{6(1-k^2)^3}}$  \\ \hline
8 &-13 & 8 & 9 & 0 & 2 & 2 & 0 & -4 & $\displaystyle{\frac{-1+12k-30k^2+20k^3+15k^4}{30(1-4k-k^2)^3}}$  \\ \hline\hline
9 &14 & -19 & -6 & 15 & 4 & 0 & -5 & 1 & $\displaystyle{\frac{2+8k+4k^2-16k^3-5k^4}{20(1+k-k^2)^4}}$  \\ \hline
10 &-6 & 15 & 14 & -19 & -4 & 1 & 3 & 0 &$\displaystyle{\frac{-(1+k^2)(1+4k-3k^2-36k^3+3k^4+4k^5-k^6)}{4k^4}}$ \\ \hline
11& 15 & -6 & -19 & 14 & 2 & -5 & 0 & 3 & $\displaystyle{\frac{-7+29k^2-8k^3-21k^4-8k^5-k^6}{2(1-k^2)^4}}$ \\ \hline
12 &-19 & 14 & 15 & -6 & 1 & 3 & 1 & -5 & $\displaystyle{ \frac{1-6k+37k^2-98k^3+35k^4+110k^5+55k^6+10k^7}{10(1-4k-k^2)^4}}$ \\ \hline
\end{tabular}}
\label{Table:8}
\end{table}

\begin{table}[H]
\caption{Data for Conj.~\ref{Conj:1} and Th.~\ref{T:5}. Here $m$ is a nonnegative integer.}
{\renewcommand{\arraystretch}{1.9}
\centering
\begin{tabular}{|c||c|c|c|c||c|c|c|c|}
\hline
label &$e_1$ & $e_2$ & $e_5$ & $e_{10}$ & $a_0$ & $a_1$ & $a_2$ & $a_3$  \\
\hline \hline
$1.m$ &$10m+8$ & $-10m-7$ & $-2m$ & $2m+3$ & 1 & 0 & $-2m-2$ & $2m+1$ \\ \hline
$2.m$ & $-2m$ & $2m+3$ & $10m+8$ & $-10m-7$ & $-2m-1$ & $2m+1$ & 0 & 0 \\ \hline
$3.m$& $2m+3$ & $-2m$ & $-10m-7$ & $10m+8$ & $2m+2$ & $-2m-2$ & 0 & 0 \\ \hline
$4.m$ & $-10m-7$ & $10m+8$ & $2m+3$ & $-2m$ & 1 & 0 & $2m+1$ & $-2m-2$ \\ \hline
\end{tabular}}
\label{Table:9}
\end{table}

We find there are essentially three distinct sets of exponents, given by
$$
(e_1,e_2,e_5,e_{10}) = (8,-7,0,3),\; (8,-13,0,9) \quad\text{and}\quad (14,-19,-6,15).
$$
Each of the three sets of exponents occurs in four permutations
$$
(e_1,e_2,e_5,e_{10}),\;(e_5,e_{10},e_1,e_{2}),\;(e_{10},e_5,e_2,e_{1}),\;(e_2,e_1,e_{10},e_{5}).
$$
The first set $(e_1,e_2,e_5,e_{10}) = (8,-7,0,3)$, along with its permutations, appears to generalise to an infinite family that
is detailed in~Table~\ref{Table:9}.  In that table, $m$ is any nonnegative integer. When $m=0$, then entries labelled
$1.m$, $2.m$, $3.m$ and $4.m$ in Table~\ref{Table:9} reduce to the entries labelled $1.0$, $2.0$, $3.0$ and $4.0$, respectively, in Table~\ref{Table:8}.

For each set of exponents in Table~\ref{Table:8} or Table~\ref{Table:9}, let
\begin{equation}
\label{geq}
g(k) = \int f(k) \,\frac{\ud k}{k} = \int k^{a_0-1}(1-k^2)^{a_1}(1+k-k^2)^{a_2}(1-4k-k^2)^{a_3} \, \ud k.
\end{equation}
The functions $g(k)$ in Table~\ref{Table:8} have been determined by evaluating the integrals~\eqref{geq}. They are all rational functions of~$k$.

We then conducted another computer search over integers $a_0$, $a_1$, $a_2$, $a_3$ with $a_0+a_1+a_2+a_3=0$ to search for when
the integral $g(k)$ in~\eqref{geq} is a rational function of $k$, i.e., no logarithm terms are involved. After testing the range
$$
-40 \leq a_1,\,a_2,\,a_3 \leq 40
$$
we are led to the following conjecture.

\begin{conjecture}
\label{Conj:1}
Suppose $a_1$, $a_2$ and $a_3$ are integers. Then
$$
\int \frac{(1-k^2)^{a_1}(1+k-k^2)^{a_2}(1-4k-k^2)^{a_3}}{k^{a_1+a_2+a_3+1}}\, \ud k
$$
is a rational function of $k$ if and only if
 $(a_1,a_2,a_3)$ is one of the 12 sets of parameters in Table~\ref{Table:8} or 
 belongs to one of the four infinite families in Table~\ref{Table:9}.
\end{conjecture}
\noindent
As a trivial example, if $a_1=a_2=a_3=0$ then (ignoring the arbitrary constant~$+C$) the integral is
$$
\int \frac{1}{k} \, \ud k = \log k 
$$
and this is not a rational function of $k$. This corresponds to the triple $(a_1,a_2,a_3)= (0,0,0)$ not occurring in either of Tables~\ref{Table:8} or~\ref{Table:9}.
\\
As a second example, if $a_1=0$, $a_2 = -2$ and $a_3=1$, then the integral is
$$
\int \frac{1-4k-k^2}{(1+k-k^2)^2} \ud k = -\left(\frac{1+k^2}{1+k-k^2}\right)
$$
and this is a rational function of $k$. This corresponds to the data labelled 1.0 in Table~\ref{Table:8}.

We will now prove that for the data in Tables~\ref{Table:8} and~\ref{Table:9}, the integral is a rational function of~$k$; that is, we prove the ``if'' part of the conjecture.
As we will see in the proof, some of the cases are not completely trivial.

\begin{theorem}
\label{T:5}
Suppose  $(a_1,a_2,a_3)$ is one of the 12 sets of parameters in Table~\ref{Table:8} or 
 belongs to one of the four infinite families in Table~\ref{Table:9}.
 Then
$$
\int \frac{(1-k^2)^{a_1}(1+k-k^2)^{a_2}(1-4k-k^2)^{a_3}}{k^{a_1+a_2+a_3+1}}\, \ud k
$$
is a rational function of~$k$.
\end{theorem}
\begin{proof}
For the entries in Table~\ref{Table:8}, the integrals are given in the last column and are clearly rational functions of~$k$.
It remains to consider the entries in Table~\ref{Table:9}.

The integral labelled $2.m$ is trivial, because by the binomial theorem and term-by-term integration we have
\begin{align*}
\int \frac{(1-k^2)^{2m+1}}{k^{2m+1}} \frac{\ud k}{k}
&= \sum_{j=0}^{2m+1} {2m+1 \choose j} (-1)^j \int k^{2j-2m-2} \ud k \\
&= \sum_{j=0}^{2m+1} {2m+1 \choose j} (-1)^j  \frac{k^{2j-2m-1}}{2j-2m-1} \\
&= \sum_{j=0}^{m} {2m+1 \choose j} \frac{(-1)^{j+1}}{2m+1-2j} \left(\frac{1}{k^{2m+1-2j}}+k^{2m+1-2j}\right),
\end{align*}
and this is clearly a rational function of $k$.

The integral labelled $3.m$ is similarly easy. On making the change of variable $u=1-k^2$ we have
$$
\int \frac{k^{2m+1}}{(1-k^2)^{2m+2}}\,\ud k = -\frac12\int \frac{(1-u)^m}{u^{2m+2}}\,\ud u
$$
and after expanding the numerator by the binomial theorem, the integral is clearly a rational function of $u$, and hence a rational function of $k$.

This leaves the cases labelled $1.m$ and $4.m$, that is, the integrals
$$
\int \frac{(1-4k-k^2)^{2m+1}}{(1+k-k^2)^{2m+2}}\,\ud k \quad\text{and}\quad \int \frac{(1+k-k^2)^{2m+1}}{(1-4k-k^2)^{2m+2}}\,\ud k,
$$
respectively, where $m$ is a nonnegative integer. For convenience, we multiply by $-1$ and shift the index~$m$ by~$1$. Thus, we will show that
\begin{equation}
\label{1.m4.m}
\int \frac{(k^2+4k-1)^{2m-1}}{(k^2-k-1)^{2m}}\,\ud k \quad\text{and}\quad \int \frac{(k^2-k-1)^{2m-1}}{(k^2+4k-1)^{2m}}\,\ud k
\end{equation}
are rational functions of $k$ for any positive integer~$m$.
Let 
$$
\alpha=\frac{1+\sqrt{5}}{2},\quad\beta=\frac{1-\sqrt{5}}{2},\quad\gamma=-2+\sqrt{5} \quad\text{and}\quad \delta=-2-\sqrt{5}
$$
so that
$$
(k-\alpha)(k-\beta) = k^2-k-1\quad\text{and}\quad (k-\gamma)(k-\delta) = k^2+4k-1.
$$
It is easy to check that $\alpha$, $\beta$, $\gamma$ and $\delta$ satisfy the relation
\begin{equation}
\label{abcd}
(\alpha+\beta)(\gamma+\delta) = 2(\alpha\beta+\gamma\delta),
\end{equation}
a fact that will be essential in the calculations below.

By the theory of partial fractions, the first integrand in~\eqref{1.m4.m} is
$$
\frac{(k^2+4k-1)^{2m-1}}{(k^2-k-1)^{2m}} = \frac{(k-\gamma)^{2m-1}(k-\delta)^{2m-1}}{(k-\alpha)^{2m}(k-\beta)^{2m}}
= \sum_{j=1}^{2m} \frac{a_j}{(k-\alpha)^j} + \sum_{j=1}^{2m} \frac{b_j}{(k-\beta)^j}
$$
for some constants $a_j$ and $b_j$. The integral will be a rational function of $k$ if and only if $a_1=b_1=0$.
By the theory of residues, we have
\begin{align*}
a_1 &= \text{Res}\left(\frac{(k-\gamma)^{2m-1}(k-\delta)^{2m-1}}{(k-\alpha)^{2m}(k-\beta)^{2m}}; k=\alpha\right) \\
&=\frac{1}{(2m-1)!}\lim_{k\rightarrow \alpha} \frac{\ud^{2m-1}}{\ud k^{2m-1}} \left( \frac{(k-\gamma)^{2m-1}(k-\delta)^{2m-1}}{(k-\beta)^{2m}}\right).
\end{align*}
If we appeal to Theorem~\ref{L:1} that will be proved in the next section, then we obtain~$a_1=0$. By symmetry, we also deduce $b_1=0$, 
and it follows that the first integral in~\eqref{1.m4.m} is a rational function of $k$. By symmetry of argument, i.e., interchange $(\alpha,\beta)$ with $(\gamma,\delta)$
in the calculations above, we deduce that the other integral in~\eqref{1.m4.m} is also a rational function of $k$.

This completes the proof of the theorem, and hence proves the ``if'' part of Conjecture~\ref{Conj:1}.
\end{proof}

\begin{corollary}
Let $(e_1,e_2,e_5,e_{10})$ be any of the sets of exponents in Tables~\ref{Table:8} or~\ref{Table:9} and let $a_0=\frac{1}{24} (e_1+2e_2+5e_5+10e_{10}).$
Then
$$
\int q^{a_0} E(q)^{e_1}E(q^2)^{e_2}E(q^5)^{e_5}E(q^{10})^{e_{10}} \,\frac{\ud q}{q} \quad \text{is a rational function of $k$.}
$$
If $a_0\geq 1$ (i.e., omitting the cases labelled {\rm{2.0, 6, 10, $2.m$}} in the tables) then
$$
\int_0^{e^{-2\pi/\sqrt{10}}} q^{a_0} E(q)^{e_1}E(q^2)^{e_2}E(q^5)^{e_5}E(q^{10})^{e_{10}} \,\frac{\ud q}{q} \in \mathbb{Q}\left(\sqrt{10+4\sqrt{5}} - 2 - \sqrt{5}\right).
$$
\end{corollary}
Here is an explicit example, corresponding to the label 3.1 (i.e., $m=1$ in the row labelled $3.m$) in Table~\ref{Table:9}:
$$
\int_0^{e^{-2\pi/\sqrt{10}}} q^3\frac{E(q)^5E(q^{10})^{18}}{E(q^2)^2E(q^5)^{17}}\, \ud q 
= \frac{1}{24} + \frac{(15\sqrt{5}-34)\sqrt{10+4\sqrt{5}}}{48}.
$$

\section{Limits}
\label{S:6}
In this section we evaluate the following limit that was required in the proof of Theorem~\ref{T:5} and which appears to be of some independent interest.
\begin{theorem}
\label{L:1}
Suppose $\alpha$, $\beta$, $\gamma$ and $\delta$ are distinct complex numbers that satisfy the relation
\begin{equation}
\label{E:rel}
(\alpha+\beta)(\gamma+\delta) = 2(\alpha\beta+\gamma\delta).
\end{equation}
Then
\begin{align*}
\lim_{k\rightarrow \alpha}& \frac{\ud^{n}}{\ud k^{n}} \frac{(k-\gamma)^{n}(k-\delta)^{n}}{(k-\beta)^{n+1}} \\
&= \begin{cases}
 0 & \text{if $n$ is odd} \\
\displaystyle{ (-1)^{n/2} \frac{(\gamma-\delta)^n}{(\alpha-\beta)^{n+1}}\,1^2\cdot 3^2 \cdot 5^2 \cdots (n-1)^2} & \text{if $n$ is even.}
 \end{cases}
\end{align*}
\end{theorem}

We will require two lemmas. The first is a technical result on the simplification of a series.

\begin{lemma}
\label{L:L1}
Let $n$ be a positive integer and $z$ be a complex variable. Let
\begin{equation}
\label{g1}
g(z) = \sum_{t=0}^{n} {t+n \choose t} \left(\frac{-1}{2}\right)^t (1+z)^t
\sum_{s=0}^{n-t} {n \choose s}{n \choose s+t}z^s.
\end{equation}
Then
$$
g(z) = \begin{cases}
0 & \text{if $n$ is odd,}\\
\displaystyle{\frac{(-1)^{n/2}}{2^n} {n \choose n/2} (z-1)^n} & \text{if $n$ is even.}
\end{cases}
$$
\end{lemma} 
\begin{proof}
Interchange the order of summation and use the properties
$$
{n \choose s+t} = {n \choose s} \times \frac{(-1)^t(s-n)_t}{(s+1)_t} \quad\text{and}\quad
{t+n \choose t} = \frac{(n+1)_t}{t!}
$$
to deduce
\begin{align}
g(z) &= \sum_{s=0}^{n} {n \choose s}^2 z^s \; \sum_{t=0}^{n-s} \frac{(n+1)_t (s-n)_t}{t!(s+1)_t} \left( \frac{1+z}{2}\right)^t \nonumber \\
&= \sum_{s=0}^{n} {n \choose s}^2 z^s \; {}_2F_1 \left( {{n+1,\; s-n} \atop {s+1}} ; \frac{1+z}{2}\right). \label{g2}
\end{align}
On replacing $s$ with $n-s$ this can be rewritten
\begin{equation}
\label{g3}
g(z) = z^{n}\sum_{s=0}^{n} {n \choose s}^2 z^{-s} \; {}_2F_1 \left( {{n+1,\; -s} \atop {n+1-s}} ; \frac{1+z}{2}\right).
\end{equation}
By setting $c=b-k$ in~\cite[Cor. 2.3.3]{aar}, where $k\in \mathbb{Z}^+$, we have
$$
{}_2F_1 \left( {{a,\;b} \atop {b-k}} ; z\right)
= \frac{(-1)^k(a)_k}{(1-b)_k} (1-z)^{-a-k} {}_2F_1 \left( {{-k,\; b-a-k} \atop {1-a-k}} ; 1-z\right),
$$
and using this in~\eqref{g3}
we obtain the equivalent form
\begin{equation}
\label{g4}
g(z) = z^{n}\sum_{s=0}^{n} {n \choose s} (-z)^{-s} \; {}_2F_1 \left( {{n+1,\; -s} \atop {1}} ; \frac{1-z}{2}\right).
\end{equation}
Interchange the order of summation to obtain
\begin{align*}
g(z) &= z^{n}\sum_{s=0}^{n} {n \choose s} (-z)^{-s} \sum_{j=0}^s \frac{(n+1)_j(-s)_j}{j!^2}\left(\frac{1-z}{2}\right)^j \\
&= z^{n}\sum_{j=0}^{n} \frac{(n+1)_j}{j!^2}\left(\frac{1-z}{2}\right)^j \sum_{s=j}^n {n \choose s}(-s)_j (-z)^{-s}.
\intertext{Shift the summation index by letting $k=s-j$ and simplify the result, to obtain}
g(z) &= z^{n}\sum_{j=0}^{n} \frac{(n+1)_j}{j!^2}\left(\frac{1-z}{2}\right)^j \;\sum_{k=0}^{n-j} {n \choose j+k}(-j-k)_j(-1)^{j+k} z^{-j-k}  \\
&= z^{n}\sum_{j=0}^{n} \frac{(n+1)_j}{j!^2}\left(\frac{1-z}{2}\right)^j \;\sum_{k=0}^{n-j} {n \choose j+k}\frac{(j+k)!}{k!} (-1)^{k} z^{-j-k}  \\
&= z^{n}\sum_{j=0}^{n} \frac{(n+1)_j}{j!^2}\left(\frac{1-z}{2}\right)^j\,(-1)^j\,z^{-j}\, \frac{n!}{(n-j)!}\, \sum_{k=0}^{n-j} {n-j \choose k} (-1)^{k} z^{-k}.
\intertext{The inner sum can be evaluated by the binomial theorem to give}
g(z)&= z^{n}\sum_{j=0}^{n} \frac{(-n)_j(n+1)_j}{j!^2}\left(\frac{1-z^{-1}}{2}\right)^j \times (1-z^{-1})^{n-j}
\end{align*}
and this simplifies to
\begin{equation}
\label{gfinal}
g(z) = (z-1)^n {}_2F_1 \left( {{-n,\; n+1} \atop {1}} ; \frac12\right).
\end{equation}
The value of this ${}_2F_1$ is a special case of Gauss' result~\cite[Th. 3.5.4]{aar}
$$
{}_2F_1 \left( {{a,\; b} \atop {\frac{a+b+1}{2}}} ; \frac12\right) = \frac{\Gamma(\frac12)\Gamma(\frac{a+b+1}{2})}{\Gamma(\frac{a+1}{2})\Gamma(\frac{b+1}{2})}.
$$

When $a=-2m$ is a negative even integer, Gauss' result simplifies to
$$
{}_2F_1 \left( {{-2m,\; b} \atop {-m+\frac{b+1}{2}}} ; \frac12\right) = \frac{\Gamma(\frac12)\Gamma(-m+\frac{b+1}{2})}{\Gamma(-m+\frac{1}{2})\Gamma(\frac{b+1}{2})}
= \frac{(\frac12-m)_m}{(\frac{b+1}{2}-m)_m} = (-1)^m\frac{(\frac12)_m}{(\frac{b+1}{2}-m)_m}.
$$
Hence, for $b=2m+1$ we obtain
$$
{}_2F_1 \left( {{-2m,\; 2m+1} \atop {1}} ; \frac12\right) = 
(-1)^m\frac{(\frac12)_m}{m!} = \frac{(-1)^m}{2^{2m}} {2m \choose m}.
$$
On substituting this result back into \eqref{gfinal} we obtain the result in the case when $n$ is an even integer.

In the limiting case when $a$ approaches a negative odd integer, i.e.,  \mbox{$a\rightarrow -2m-1$,}
Gauss' result evaluates to 0 because of the pole of the gamma function that occurs in the denominator on the right
hand side.  This completes the proof.
\end{proof}

Here is the other lemma that is required.
\begin{lemma}
\label{T:limit}
Suppose $a$ and $b$ are non-zero constants and $a\neq b$. Let $n$ be a positive integer. 
Let
$$
f(a,b) =  \lim_{x\rightarrow \frac{2ab}{a+b}} \frac{\ud^{n}}{\ud x^{n}} \frac{(x-a)^{n}(x-b)^{n}}{x^{n+1}}.
$$
Then
$$
f(a,b) = \begin{cases}
 0 & \text{if $n$ is odd} \\
\displaystyle{ (-1)^{n/2} (a-b)^n \left(\frac{a+b}{2ab}\right)^{n+1} 1^2\cdot 3^2 \cdot 5^2 \cdots (n-1)^2} & \text{if $n$ is even.}
 \end{cases}
$$
\end{lemma}

\begin{proof}
By the binomial theorem,
$$
f(a,b)= \lim_{x\rightarrow \frac{2ab}{a+b}} \frac{\ud^{n}}{\ud x^{n}} \sum_{i=0}^{n}\sum_{j=0}^{n} {n \choose i}{n \choose j} (-a)^i(-b)^j  x^{n-1-i-j}.
$$
The powers of $x$ in the sum are $x^{n-1}$, $x^{n-2},\ldots,x,\,x^0,\,x^{-1},\,x^{-2},\ldots,x^{-n-1}$. Only the negative powers of $x$ are not annihilated by the differentiation, that is, only
the terms with $i+j \geq n$, remain. Therefore,
\begin{align*}
f(a,b) &= -\lim_{x\rightarrow \frac{2ab}{a+b}}  \sum_{i+j\geq n} {n \choose i}{n \choose j} \frac{(i+j)!}{(i+j-n)!} (-a)^i(-b)^j  x^{-1-i-j} \\
&=- \sum_{i+j\geq n} {n \choose i}{n \choose j} \frac{(i+j)!}{(i+j-n)!} \frac{(-1)^{i+j}(a+b)^{i+j+1}}{2^{i+j+1}a^{j+1}b^{i+1}}.
\intertext{Put $t=i+j-n$ to deduce}
f(a,b) &= \sum_{t=0}^{n} \sum_{j=t}^{n} {n \choose j-t}{n \choose j} \frac{(t+n)!}{t!} \frac{(-1)^t(a+b)^{t+n+1}}{2^{t+n+1}a^{j+1}b^{n+1+t-j}}.
\intertext{Now make the further change of summation index $s=j-t$ and put $z=b/a$ to obtain}
f(a,b) &= \sum_{t=0}^{n} \sum_{s=0}^{n-t} {n \choose s}{n \choose s+t} \frac{(t+n)!}{t!} \frac{(-1)^t(a+b)^{t+n+1}}{2^{t+n+1}a^{s+t+1}b^{n+1-s}} \\
&= \frac{n!(a+b)^{n+1}}{2^{n+1}ab^{n+1}}
\sum_{t=0}^{n} {t+n \choose t} \frac{(-1)^t(a+b)^t}{2^ta^t}
\sum_{s=0}^{n-t} {n \choose s}{n \choose s+t}\frac{b^s}{a^s} \\
&= \frac{n!(a+b)^{n+1}}{2^{n+1}ab^{n+1}} \,g\left(\frac{b}{a}\right),
\end{align*}
where $g(z)$ is as defined by~\eqref{g1}. Now apply Lemma~\ref{L:L1} to complete the proof.
\end{proof}

Now we are ready to prove Theorem~\ref{L:1}.
\begin{proof}[Proof of Theorem~\ref{L:1}]
Let $x=k-\beta$ to deduce
$$
\lim_{k\rightarrow \alpha} \frac{\ud^{n}}{\ud k^{n}} \frac{(k-\gamma)^{n}(k-\delta)^{n}}{(k-\beta)^{n+1}} 
= \lim_{x\rightarrow \alpha-\beta} \frac{\ud^{n}}{\ud x^{n}} \frac{(x+\beta-\gamma)^{n}(x+\beta-\delta)^{n}}{x^{n+1}}.
$$
The relation~\eqref{E:rel} can be rearranged to give
$$
\alpha - \beta = \frac{2(\gamma-\beta)(\delta-\beta)}{(\gamma-\beta)+(\delta-\beta)}.
$$
Hence, Lemma~\ref{T:limit} can be applied with $a=\gamma-\beta$ and $b=\delta-\beta$, and noting that
$$
\frac{a+b}{2ab} = \frac{1}{\alpha-\beta}.
$$
The proof is completed by substituting the values above into the result of Lemma~\ref{T:limit}.
\end{proof}

\section{Appendix}
\label{S:7}
We evaluate $k(e^{-2\pi/\sqrt{10}})$. By \eqref{rp1} and~\eqref{rp3} we have
$$
\left(\frac{k}{1+k-k^2}\right)\left(\frac{1-k^2}{1-4k-k^2}\right)^2 = \frac{\eta^6(5\tau)}{\eta^6(\tau)}.
$$
Now put $\tau = i/\sqrt{10}$ so that $q=e^{-2\pi/\sqrt{10}}$ and apply the transformation formula for the eta function~\eqref{E:eta} to deduce
\begin{equation}
\label{E:X1}
\left.\left(\frac{k}{1+k-k^2}\right)\left(\frac{1-k^2}{1-4k-k^2}\right)^2\right|_{q=e^{-2\pi/\sqrt{10}}}
= \left(\frac{\eta(\frac{5i}{\sqrt{10}})}{\eta(\frac{i}{\sqrt{10}})}\right)^6 
=\frac{1}{125} \left(\frac{\eta(\frac{2i}{\sqrt{10}})}{\eta(\frac{10i}{\sqrt{10}})}\right)^6 .
\end{equation}
On the other hand, from~\eqref{rp2} and~\eqref{rp4} we have
$$
\left(\frac{1+k-k^2}{k}\right)^2\left(\frac{1-4k-k^2}{1-k^2}\right) = \frac{\eta^6(2\tau)}{\eta^6(10\tau)}
$$
and setting $\tau = i/\sqrt{10}$ gives
\begin{equation}
\label{E:X2}
\left.\left(\frac{1+k-k^2}{k}\right)^2\left(\frac{1-4k-k^2}{1-k^2}\right)\right|_{q=e^{-2\pi/\sqrt{10}}} = \left(\frac{\eta(\frac{2i}{\sqrt{10}})}{\eta(\frac{10i}{\sqrt{10}})}\right)^6.
\end{equation}
Combining~\eqref{E:X1} and~\eqref{E:X2} gives
$$
\left.\left(\frac{k}{1+k-k^2}\right)\left(\frac{1-k^2}{1-4k-k^2}\right)^2\right|_{q=e^{-2\pi/\sqrt{10}}}
= \frac{1}{125} \left.\left(\frac{1+k-k^2}{k}\right)^2\left(\frac{1-4k-k^2}{1-k^2}\right)\right|_{q=e^{-2\pi/\sqrt{10}}}
$$
and this simplifies to
$$
\left.\left(\frac{k}{1+k-k^2}\right)\left(\frac{1-k^2}{1-4k-k^2}\right)\right|_{q=e^{-2\pi/\sqrt{10}}} = \frac{1}{5}.
$$
This can be simplified further by putting $u=\frac{1}{k}-k$ to get
$$
\left(\frac{1}{u+1}\right)\left(\frac{u}{u-4}\right) = \frac15.
$$
Since $0<k<\sqrt{5}-2$ for $0<q<1$ by~\cite[p. 524]{cooperbook} 
we choose the positive solution $u = 4+2\sqrt{5}$. Hence,
\begin{equation}
\label{u}
\left[ \frac{1}{k}-k \right]_{q=e^{-2\pi/\sqrt{10}}} \; = 4+2\sqrt{5} .
\end{equation}
Solving for the positive value of $k$  gives
\begin{equation}
\label{keval}
k(e^{-2\pi/\sqrt{10}}) = \sqrt{10+4\sqrt{5}} - 2 - \sqrt{5}.
\end{equation}

\begin{remark}
Although evaluations of $k(q)$ have been considered by Y. Lee and Y. K. Park~\cite{lee} by other methods, 
they do not seem to have calculated the particular value in~\eqref{keval}.
\end{remark}


\begin{thebibliography}{99}

\bibitem{madhu}
C. Adiga, T. Kim, M. S. Mahadeva Naika and H. S. Madhusudhan, 
{\em On Ramanujan's cubic continued fraction and explicit evaluations of theta-functions,}
Indian J. Pure Appl. Math., {\bf 35} (2004), 1047--1062.

\bibitem{ahlgren}
S. Ahlgren, B. C. Berndt, A. J. Yee and A. Zaharescu, 
{\em Integrals of Eisenstein series and derivatives of L-functions,}
IMRN, {\bf 32} (2002), 1723--1738.

\bibitem{ksw2009}
A. Alaca, \c{S}. Alaca and K. S. Williams,
{\em Some infinite products of Ramanujan type,}
Canadian Math. Bull., {\bf 52} (2009), 481--492.

\bibitem{andrews}
G. E. Andrews,
{\em Ramanujan's ``lost'' notebook. III. The Rogers-Ramanujan continued fraction,}
Adv. Math., {\bf 41}(1981), 186--208.

\bibitem{aar}
G. E. Andrews, R. Askey and R. Roy,
{\em Special Functions,}
Cambridge University Press, Cambridge, 1999.

\bibitem{hitparade}
G. E. Andrews and B. C. Berndt,
{\em Your hit parade: the top ten most fascinating formulas in Ramanujan's lost notebook,}
Notices Amer. Math. Soc., {\bf 55} (2008), 18--30.

\bibitem{aygin}
Z. S. Aygin and P. C. Toh,
{\em When is the derivative of an eta quotient another eta quotient?}
J. Math. Anal. Appl. {\bf 480} (2019), no. 1, 123366, 22 pp.

\bibitem{Part3}
B. C. Berndt,
{\em Ramanujan's Notebooks, Part III},
Springer-Verlag, New York, 1991.

\bibitem{spirit} B. C. Berndt,
{\em Number Theory in the Spirit of Ramanujan,}
Amer. Math. Soc., Providence, RI, 2006.

\bibitem{hhc}
H. H. Chan,
{\em Theta Functions, Elliptic Functions and $\pi$,}
De Gruyter, Berlin, 2020.

\bibitem{coopernonic}
S. Cooper,
{\em A simple proof of an expansion of an eta-quotient as a Lambert series,}
Bull. Austral. Math. Soc., {\bf 71} (2005), 353--358.

\bibitem{cooperbook}
S. Cooper,
{\em Ramanujan's Theta Functions,}
Springer, Cham, 2017.

\bibitem{4term}
S. Cooper,
{\em Ap{\'e}ry-like sequences defined by four-term recurrence relations,}
to appear.

\bibitem{cooperlam}
S. Cooper and H. Y. Lam,
{\em Sums of two, four, six and eight squares and triangular numbers: an elementary approach,}
Indian J. Math., {\bf 44} (2002),  21--40.

\bibitem{cooperye12}
S. Cooper and D. Ye,
{\em The level $12$ analogue of Ramanujan's function $k$,}
 J. Aust. Math. Soc., {\bf 101} (2016), 29--53.

\bibitem{doyle}
G. Doyle and K. S. Williams,
{\em Evaluation of some $q$-integrals in terms of the Dedekind eta function,} Analysis {\bf 38} (2018), 63--79.

\bibitem{fine}
N. Fine,
{\em Basic Hypergeometric Series and Applications,}
Amer. Math. Soc., Providence, RI, 1988.

\bibitem{hirschhorn4sq}
M. D. Hirschhorn,
{\em A simple proof of Jacobi's four-square theorem,}
Proc. Amer. Math. Soc., {\bf 101} (1987), 436--438.

\bibitem{hirschhorn}
M. D. Hirschhorn,
{\em The Power of $q$,}
Springer, Cham, 2017.

\bibitem{johnson}
W. P. Johnson,
{\em An Introduction to $q$-analysis,}
Amer. Math. Soc., Providence, RI, 2020.

\bibitem{lee2}
Y. Lee and Y. K. Park,
{\em A continued fraction of order twelve as a modular function,}
Math. Comp., {\bf 87} (2018), 2011--2036.

\bibitem{lee}
Y. Lee and Y. K. Park,
{\em Ramanujan's function $k(\tau)=r(\tau)r^2(2\tau)$ and its modularity,}
Open Math., {\bf 18} (2020), 1727--1741.

\bibitem{lin}
B. L. S. Lin,
{\em On the expansion of a continued fraction of order $12$},
Int. J. Number Theory, {\bf 9} (2013), 2019--2031.

\bibitem{naika2012}
M. S. Mahadeva Naika, S. Chandankumar and K. Sushan Bairy,
{\em Some new identities for a continued fraction of order $12$,} 
South East Asian J. Math. Math. Sci., {\bf 10} (2012), 129--140. 

\bibitem{naika}
M. S. Mahadeva Naika, B. N. Dharmendra and K. Shivashankara,
{\em A continued fraction of order twelve,}
Cent. Eur. J. Math., {\bf 6} (2008), 393--404.


\bibitem{lost}
S. Ramanujan,
{\em The Lost Notebook and Other Unpublished Papers,}
Narosa, New Delhi, 1988.

\bibitem{vasuki}
K. R. Vasuki, Abdulrawf A. A. Kahtan, G. Sharath and C. Sathish Kumar,
{\em On a continued fraction of order $12$,}
Ukrainian Math. J., {\bf 62} (2011), 1866--1878. 

%

\bibitem{zhang}
L.-C. Zhang,
{\em Some $q$-integrals associated with modular forms,}
J. Math. Anal. Appl., {\bf 150} (1990), 264--273.


\end{thebibliography}
\end{document}